\documentclass[11 pt]{amsart}
\usepackage{amscd,amsfonts,amssymb,amsmath}
\usepackage{hyperref}
\usepackage{epsfig}
\usepackage{mathtools}
\usepackage{mathrsfs}
\usepackage{tabu}
\usepackage{tikz-cd}
\usepackage {latexsym}
\usepackage{tikz}
\newtheorem{theorem}{Theorem}[section]
\newtheorem{corollary}[theorem]{Corollary}
\newtheorem{lemma}[theorem]{Lemma}
\newtheorem{proposition}[theorem]{Proposition}
\theoremstyle{definition}

\newtheorem{problem}[theorem]{Problem}

\newtheorem{remark}[theorem]{Remark}
\numberwithin{equation}{subsection}
\newtheorem*{ack}{Acknowledgement}
\usepackage[all,cmtip]{xy}
\usepackage{graphicx} 

\newcommand{\Spe}{\operatorname{Spe}}
\newcommand{\Aut}{\operatorname{Aut}}

\newcommand{\Spec}{\operatorname{Spe}}

\newcommand{\Inn}{\operatorname{Inn}}
\newcommand{\Out}{\operatorname{Out}}

\newcommand{\R}{\operatorname{R}}

\newcommand{\Z}{\operatorname{Z}}
\newcommand{\C}{\operatorname{C}}
\newcommand{\U}{\operatorname{U}}

\begin{document}

\title{Automorphisms of odd Coxeter groups}

\author{Tushar Kanta Naik}
\address{Department of Mathematical Sciences, Indian Institute of Science Education and Research (IISER) Mohali, Sector 81,  S. A. S. Nagar, P. O. Manauli, Punjab 140306, India.}
\email{mathematics67@gmail.com / tushar@iisermohali.ac.in}

\author{Mahender Singh}
\address{Department of Mathematical Sciences, Indian Institute of Science Education and Research (IISER) Mohali, Sector 81,  S. A. S. Nagar, P. O. Manauli, Punjab 140306, India.}
\email{mahender@iisermohali.ac.in}

\subjclass[2010]{Primary 20F55; Secondary  20F28, 20F36}
\keywords{Automorphism, braid group, Coxeter group, Hopfian property, $R_{\infty}$-property, twin group}

\begin{abstract}
An odd Coxeter group $W$ is one which admits a Coxeter system $(W,S)$ for which all the exponents $m_{ij}$ are either odd or infinity. The paper  investigates the family of odd Coxeter groups whose associated labeled graphs $\mathcal{V}_{(W,S)}$ are trees. It is known that two Coxeter groups in this family are isomorphic if and only if they admit Coxeter systems having the same rank and the same multiset of finite exponents. In particular, each group in this family is isomorphic to a group that admits a Coxeter system whose associated labeled graph is a star shaped tree. We give the complete description of the automorphism group of this group, and derive a sufficient condition for the splitting of the automorphism group as a semi-direct product of the inner and the outer automorphism groups. As applications, we prove that Coxeter groups in this family satisfy the $R_\infty$-property and are (co)-Hopfian. We compare structural properties, automorphism groups,  $\R_\infty$-property and (co)-Hopfianity of a special odd Coxeter group whose only finite exponent is three with the braid group and the twin group.
\end{abstract}

\maketitle

\section{Introduction}
A group $W$  is called a {\it Coxeter group} if it admits a  presentation of the form $\big \langle S \mid R \big \rangle$, where $S= \{ w_i~|~ i \in \Pi\}$ is the set of generators and 
\begin{equation}\label{Coxeter presentation}
R  = \big \{ (w_iw_j)^{m_{ij}}~ \mid m_{ij}\in \mathbb{N} \cup \{\infty\},~m_{ij}=m_{ji}, ~m_{ij} = 1 \Leftrightarrow i =j~\big \}
\end{equation}
is the set of defining relations. Such a generating set $S$ is called a Coxeter generating set for $W$, and the pair $(W, S)$ is called a {\it Coxeter system}. We refer to the
powers $m_{ij}$, $i \neq j$,  as the {\em exponents} of the Coxeter system $(W, S)$. The cardinality of the set $S$ is called the {\it rank} of the Coxeter system $(W,S)$. We consider only Coxeter systems of finite rank. In general, the Coxeter system $(W,S)$ is not uniquely determined by the group $W$.  For example,  the dihedral group $D_{2(2n)}$ of order $4n$, where $n\ge 3$ is an odd integer, has different Coxeter presentations as
$$\big\langle r, s \mid r^2 = s^2 = (rs)^{2n}=1 \big\rangle \cong D_{4n} \cong \big\langle x, y, z \mid x^2=y^2=z^2 = (xy)^2 = (xz)^2 = (yz)^n=1 \big\rangle.$$
\par

Graphs are indispensable tools in the study of Coxeter groups and there are several ways to associate graphs to a Coxeter system. Given a Coxeter system $(W,S)$ let $\Gamma_{(W, S)}$ be the edge-labeled graph whose vertex set is the index set $\Pi$ of elements of $S$ and there is an edge between the vertices $i$ and $j$ whenever the exponent $m_{ij}\geq 3$. Further, the edge between $i$ and $j$ is labelled by $m_{ij}$ whenever $m_{ij}\geq 4$. The Coxeter system $(W,S)$ is said to be { \it irreducible} if the graph $\Gamma_{(W, S)}$ is connected.  For each subset $I\subseteq S$, the subgroup 
$$W_I = \big\langle w_i \mid w_i\in I \big\rangle $$ of $W$ is called a {\it standard parabolic subgroup}. It is well-known that $W_I$ is itself a Coxeter group. A conjugate of a standard parabolic subgroup is called a {\it parabolic subgroup}. If the graph $\Gamma_{(W, S)}$ is disconnected, then the Coxeter group $W$ is the direct product of its standard parabolic subgroups corresponding to each connected component of the graph $\Gamma_{(W, S)}$.
\par

Existence of multiple Coxeter systems $(W,S)$ and hence multiple  graphs  $\Gamma_{(W, S)}$ to a given Coxeter group $W$ led to the following well-known isomorphism problem for finite rank Coxeter groups.  

\begin{problem}\label{prob1}
Given Coxeter graphs $\Gamma_{(W_1, S_1)}$ and $\Gamma_{(W_2, S_2)}$, determine whether the Coxeter groups  $W_1$ and $W_2$ are isomorphic. 
\end{problem}

We refer the monograph \cite{Bahls2005} and the survey articles \cite{Muhlherr2006, Nuida2014} for more details on this problem. A more general problem \cite[Problem 2]{Muhlherr2006} is

\begin{problem}\label{prob2}
Given Coxeter graphs $\Gamma_{(W_1, S_1)}$ and $\Gamma_{(W_2, S_2)}$, find all isomorphisms between the groups $W_1$ and $W_2$.
\end{problem}

A solution to Problem \ref{prob2} is equivalent to a solution to Problem \ref{prob1} and a description of the automorphism group of $W$ for any Coxeter system $(W,S)$. Study of automorphism group of Coxeter groups have a long and rich history. Tits \cite{Tits} showed that if the graph  $\Gamma_{(W, S)}$ contains no triangles, then  $\Aut(W)$ is the semi-direct product of $\Inn(W)$ and $\Aut(F)$, where $F$ is the groupoid of commuting subsets of $S$ relative to symmetric difference.
\par

A Coxeter group is called {\it graph-universal} or {\it right angled} if all the exponents $m_{ij} \in \{2,\infty\}$. The automorphism group $\Aut(W)$ of a Coxeter group $W$ acts on the set of conjugacy classes of involutions in $W$ and the kernel $\Spec(W)$ of this action  is referred as the group of {\it special automorphisms} of $W$. If $W$ is right angled, using ideas from \cite{Tits}, M\"{u}hlherr   \cite{Muhlherr1998} gave a presentation for $\Spec(W)$ in terms of natural generators. Automorphism groups of certain right angled Coxeter groups, called twin groups, have been determined in \cite{James, NaikNandaSingh1}. A necessary and sufficient criteria on the graph $\Gamma_{(W, S)}$ of a right-angled Coxeter group $W$ such that its outer automorphism group contains a finite index subgroup that surjects onto the free group $F_2$ is given in \cite{SaleSusse2019}. Further, automorphism groups of universal Coxeter groups have been investigated in a recent work of Varghese \cite{Varghese}.
\par

There is another labeled graph $\mathcal{V}_{(W,S)}$ associated to a Coxeter system $(W, S)$ whose vertex set is the index set $\Pi$ of elements of $S$ and there is an edge between the vertices $i$ and $j$ if and only if $m_{ij}< \infty$. Further, the edge between $i$ and $j$ is labelled by $m_{ij}$ whenever $m_{ij}\neq 3$. If the graph $\mathcal{V}_{(W, S)}$ is disconnected, then $W$ is a free product of its standard parabolic subgroups corresponding to each connected component of the graph $\mathcal{V}_{(W,S)}$.
\par

 An {\it even Coxeter group} is one where all the exponents $m_{ij}$ are either even or infinite. In \cite{Bahls2006}, Bahls computed $\Aut(W)$ for any even Coxeter group whose graph $\mathcal{V}_{(W,S)}$ is connected, contains no edges labeled 2, and cannot be separated into more than 2 connected components by removing a single vertex.  In \cite{Franzsen2001, Franzsen2002}, automorphism groups of Coxeter groups of rank 3 are completely determined. Automorphisms of infinite Coxeter groups of rank $n$ having a finite parabolic subgroup of rank $(n-1)$ are investigated in \cite{Franzsen2003}.
 \par 
 
 Each permutation $\alpha$ of $\Pi$ which is an automorphism of the graph $\Gamma_{(W,S)}$ extends uniquely to an automorphism of the group $W$. Such an automorphism is called a {\it graph automorphism}, and the set of all graph automorphism forms a subgroup $\Aut \big(\Gamma_{(W, S)}\big)$ of $\Aut(W)$.  A Coxeter group $W$ is called { \it rigid} if for any two Coxeter systems $(W,S)$ and $(W,S')$ the graphs $\Gamma_{(W, S)}$ and $\Gamma_{(W, S')}$ are isomorphic. If $W$ is a rigid Coxeter groups such that $\Gamma_{(W, S)}$ is a complete graph with only odd edge labels, then it is shown in \cite{Rian2007} that $\Aut(W)$ is the semi-direct product of $\Inn(W)$ and the group of graph automorphisms $\Aut \big(\Gamma_{(W, S)}\big)$. 
\par

We say that a Coxeter system $(W, S)$ is {\it odd} if all the exponents $m_{ij}$'s are either odd or infinite. An {\it odd Coxeter group} $W$ is one which admits an odd Coxeter system. Odd Coxeter groups lie in the more general family of {\it groups of large type} where each exponent is at least three. We call a Coxeter system $(W, S)$ {\it connected} if the graph $\mathcal{V}_{(W, S)}$ is connected. The graph $\Gamma_{(W, S)}$ of an odd Coxeter group is always a complete graph containing $\mathcal{V}_{(W, S)}$ as its subgraph.
\par

In this paper, we focus on the family $\mathcal{W}$ of Coxeter groups that admit an odd connected Coxeter system of rank greater than one, and we would always work with such a system unless mentioned otherwise.  Let $\mathcal{TW}$ denote the subfamily of $\mathcal{W}$ consisting of groups $W$ that admit an odd connected Coxeter system $(W,S)$ for which the graphs $\mathcal{V}_{(W, S)}$ are trees. Solution of Problem \ref{prob1} for the family $\mathcal{TW}$ is known due to \cite{MuhlherrWeidmann2000, Brady2002}, wherein it is proved that two groups in $\mathcal{TW}$ are isomorphic if and only if they have odd connected Coxeter systems of the same rank and the same multiset of finite exponents.
\medskip

The paper is organised as follows. Section \ref{prelim} contains some preliminary results that are used in the subsequent sections. In Section \ref{automorphism section}, we give the complete description of the automorphism group of a group $\mathbb{W} \in \mathcal{TW}$ of given rank and given multiset of odd exponents such that its associated graph $\mathcal{V}_{(\mathbb{W}, \mathbb{S})}$ is a star shaped tree (Theorem \ref{splitting-by-graph-auto}).  Solution of Problem \ref{prob1} for the family $\mathcal{TW}$ due to \cite{MuhlherrWeidmann2000, Brady2002} and Theorem \ref{splitting-by-graph-auto} then give a complete solution of Problem \ref{prob2} for the family $\mathcal{TW}$. 
In Section \ref{splitting-sequences}, we give a sufficient condition for the splitting of the natural short exact sequence  $$1\rightarrow \Inn(\mathbb{W})\rightarrow \Aut(\mathbb{W}) \rightarrow \Out(\mathbb{W})\rightarrow 1$$ of automorphism groups (Theorem \ref{splitting-out-sequence}). In Section \ref{comparision}, we prove that groups in the family $\mathcal{TW}$ satisfy the $R_\infty$-property (Theorem \ref{thm-R-infinity}) and are (co)-Hopfian (Theorem \ref{thm-cohopfian}). The section then discusses a special Coxeter group $L_n \in \mathcal{TW}$ whose only finite exponent is three and is closely related to the well-known braid group $B_n$ and the twin group $T_n$, the planar analogue of $B_n$. We compare structural properties, automorphism groups,  $\R_\infty$-property and (co)-Hopfianity of $L_n$, $B_n$ and  $T_n$.
\medskip

\section{Preliminaries}\label{prelim}
We begin by setting some notations. For a group $G$, we denote its center, commutator subgroup, group of automorphisms, group of inner automorphisms and group of outer automorphisms by $\Z(G)$, $G'$, $\Aut(G)$, $\Inn(G)$ and  $\Out(G)$, respectively. For elements $x, y \in G$, the commutator $x^{-1}y^{-1}xy$ of $x$ and $y$ is denoted by $[x,\; y]$. The centraliser of an element $x \in G$ is denoted by $\C_G(x)$. We denote the element $y^{-1}xy$ by $x^y$ and the inner automorphism induced by $x$ by $\widehat{x}$. Cyclic group of order $n$ is denoted by $\mathbb{Z}/n\mathbb{Z}$ and the multiplicative group of integers mod $n$ that are coprime to $n$ is  denoted by  $\U_n$. The  symmetric group on $n \ge 2$ symbols is denoted by $S_n$ and the dihedral group of order $2n$ is denoted by $D_{2n}$.
\par

We now recall some well-known results on Coxeter groups. Let $(W, S)$ be a Coxeter system. Each element $w\in W$ can be written as a product of generators
$$w = w_{i_1}w_{i_2}\cdots w_{i_k}.$$
The minimal $k$ among all such expressions for $w$ is called the {\it length} of $w$ and we denote it by $\ell(w)$. 
\par

The following result about involutions in Coxeter groups will be useful \cite[p.4]{Muhlherr2006}.

\begin{lemma}\label{general criterion for involution}
Let $(W, S)$ be a Coxeter system and $w\in W$ an involution. Then there exists a finite parabolic subgroup $W_J$ and element $x\in \Z(W_J)$ with $\ell(x)> \ell(x')$ for all $x\neq x'\in W_J$ such that $w = yxy^{-1}$ for some $y\in W$.
\end{lemma}

The next result gives relation between the group of inner automorphisms and the group of graph automorphisms of infinite Coxeter groups \cite[p.34, Lemma 2.14]{Franzsen-thesis}.

\begin{lemma}\label{graph-intersection-inner}
Let $W$ be an infinite Coxeter group with no finite irreducible components and $(W, S)$ its Coxeter system. Then $$\Aut\big(\Gamma_{(W, S)}\big) \cap \Inn(W) = 1.$$ 
\end{lemma}

\begin{lemma}\label{trivial center}
If $W \in \mathcal{W}$, then $\Z(W) = 1$.
\end{lemma}

\begin{proof}
If $W$ is of rank two, then $W\cong D_{2m}$ for some odd integer $m \ge 3$, in which case $\Z(W) = 1$. It follows from \cite[p.22, Exercise 4]{Bjorner2005} that if rank of $W$ is greater than two, then $W$ is infinite. By \cite[p.137]{Bourbaki}, the center of an infinite irreducible Coxeter group is trivial. Since $W$ is irreducible, the assertion  follows.
\end{proof}

We need the following result concerning centralisers of Coxeter generators, which is a special case of \cite[Theorem, p.466]{Brink}.

\begin{lemma}\label{involution result of brink}
Let $W \in \mathcal{TW}$ and  $(W, S)$ its Coxeter system. Then the centraliser $\C_W(w_i )$ is a Coxeter group for each $w_i \in S$.\end{lemma}

\begin{lemma}\label{involutions are conjugate in TW_n}
If $W \in \mathcal{W}$, then any two involutions in $W$ are conjugate.
\end{lemma}

\begin{proof}
Let $(W, S)$ be an associated Coxeter system of $W$. Then by \cite[Proposition 3, Page 5]{Bourbaki} it follows that any two generators $w, w'\in S$ are conjugates of each others in $W$. Hence, to complete the proof, it suffices to show that if $w$ is an involution, then $w$ is conjugate to some generator $w_i$. By Lemma \ref{general criterion for involution}, it follows that $w$ is conjugate to some central element of some finite parabolic subgroup which is also of maximal length in that parabolic subgroup. By \cite[p.22, Exercise 4]{Bjorner2005}, all finite parabolic subgroups of $W$ are of rank one or two. A finite rank two parabolic subgroup is isomorphic to a dihedral group $D_{2m}$ for some odd integer $m\geq 3$, which has trivial center. A rank one parabolic subgroup is of the form $\langle w_i \rangle^{x}$ for some $1\leq i \leq n$ and $x\in W$, where the only non-trivial element is $w_i^{x}$. Thus, $w$ is conjugate to $w_i$.
\end{proof}

Since involutions in  $W \in \mathcal{W}$ forms a single conjugacy class, we obtain the following corollary.

\begin{corollary}
If $W \in \mathcal{W}$, then $\Spe(W) = \Aut(W)$.
\end{corollary}

The following observation is useful.

\begin{lemma}\label{centralisers of involutions}
Let $W \in \mathcal{TW}$ be a Coxeter group. If $w\in W$ is an involution, then $$\C_W(w) = \langle w \rangle.$$
\end{lemma}

\begin{proof}
By Lemma \ref{involutions are conjugate in TW_n}, it is sufficient to prove that $\C_W(w_1) = \langle w_1 \rangle$. By Lemma \ref{involution result of brink},  $\C_W(w_1)$ is a Coxeter group. Let $ x\in \C_W(w_1)\setminus \langle w_1 \rangle$ be an involution. Then $\langle w_1, x \rangle$ is a finite subgroup, in particular of order $4$. By \cite[Exercise 2d, p.130]{Bourbaki1968}, it follows that $\langle w_1, x \rangle$ is contained in some finite parabolic subgroup of $W$. But this is not possible, since all finite parabolic subgroups of $W$ are of order $2m$ for some odd integer $m$. This implies that $\C_W(w_1) = \langle w_1 \rangle$.
\end{proof}

The following result follows from \cite[Exercise 9, p.33]{Bourbaki}. Alternatively, it can be obtained by a direct application of Reidemeister-Schreier method. 
 
\begin{theorem}\label{generalised theorem for commutator}
Let $W \in \mathcal{TW}$ and $(W, S)$ be its odd connected Coxeter system of rank $n$ with $\big\{m_1, m_2, \dots, m_{n-1} \big\}$ as its multiset of finite exponents. Then
$$W' \cong \big ( \mathbb{Z}/m_1\mathbb{Z} \big ) \ast \big ( \mathbb{Z}/m_2\mathbb{Z} \big ) \ast \cdots \ast \big ( \mathbb{Z}/m_{n-1}\mathbb{Z} \big ).$$
Moreover, $W'$ consists of all even length elements of $W$ and $W = W' \rtimes \langle w_i \rangle,$ where $w_i\in S$ is any Coxeter generator.
\end{theorem}

Since the commutator subgroup of a group is characteristic, each automorphism of the group induces an automorphism of its commutator subgroup. We conclude this section with the following general result.

\begin{proposition}\label{restriciton-to-commutator}
If $W\in \mathcal{TW}$ has rank greater than two, then the restriction map $ \Aut(W) \rightarrow \Aut(W')$ is injective.
\end{proposition}

\begin{proof}
Let  $S = \{w_1, w_2, \ldots,w_n\}$, $n \ge 3$, be a Coxeter generating set of $W$. Let $\phi$ be in the kernel of the restriction map $ \Aut(W) \rightarrow \Aut(W')$. Suppose that $\phi(w_1)= x$, an involution in $W$. For each $2 \le i \le n$, since $w_1w_i \in W'$, we have $\phi (w_1w_i) = w_1w_i$, which gives $\phi (w_i) = xw_1w_i$. Since $w_iw_1 = \phi(w_iw_1) = xw_1w_i x$, we obtain $w_1w_i = w_1x(w_1w_i)xw_1 = (w_1w_i)^{xw_1}$. This implies that $xw_1 \in \C_{W}(w_1w_i)$ for each $2 \le i \le n$. Since $x$ is an involution, by Lemma \ref{involutions are conjugate in TW_n}, it must be of odd length. Thus, $xw_1$ is of even length, and hence $xw_1\in W_n'$. Thus, $xw_1 \in \C_{W'}(w_1w_i)$ for each $2 \le i \le n$. By Theorem \ref{generalised theorem for commutator}, the set $\{w_1w_2, w_1w_3, \ldots, w_1w_n \}$ generates $W'$, and hence $xw_1 \in \Z(W') = 1$. This implies that $w_1 =x= \phi(w_1)$, and consequently $\phi (w_i) = w_i$ for all $i$. Hence, $\phi$ must be the trivial automorphism of $W$.
\end{proof}

As a consequence of Lemma \ref{trivial center}, Theorem \ref{generalised theorem for commutator} and Proposition \ref{restriciton-to-commutator}, we obtain the following result.

\begin{corollary}
Let $W \in \mathcal{TW}$ and $(W, S)$ be its odd connected Coxeter system of rank $n \ge 3$ whose multiset of finite exponents is $\big\{m_1,  \dots, m_{n-1}\big\}$. Then there is a faithful representation 
$$W \hookrightarrow \Aut \big( ( \mathbb{Z}/m_1\mathbb{Z}) \ast  \cdots \ast ( \mathbb{Z}/m_{n-1}\mathbb{Z} ) \big).$$
\end{corollary}
\medskip


\section{Automorphisms of groups in $\mathcal{TW}$}\label{automorphism section}

In this section, we compute the automorphism group of groups in  the family $\mathcal{TW}$. This together with a solution of Problem \ref{prob1} due to \cite[p.1]{MuhlherrWeidmann2000} and \cite[Theorem 5.4]{Brady2002}  provides a  solution of Problem \ref{prob2} for this family. By \cite{MuhlherrWeidmann2000, Brady2002} two groups in the family $\mathcal{TW}$ are isomorphic if and only if they have the same rank and the same multiset of finite exponents. Thus, it is enough to compute the automorphism group of any group in an isomorphism class. The full automorphism group and the inner automorphism group will be the same for each group  in an isomorphism class, but the group of graph automorphisms may change depending on the graph. In view of Lemma \ref{graph-intersection-inner}, the larger the group of graph automorphisms, the simpler it is to compute the full automorphism group.

\subsection{Computation of the automorphism group} 
We choose an appropriate group in the family $\mathcal{TW}$ for computing the automorphism group. Since rank two Coxeter groups are dihedral groups whose automorphism groups are well-known, we can assume that our group has rank  greater than two. Consider the group  $\mathbb{W}$ with the following presentation
\begin{equation}\label{star-group}
\mathbb{W} = \big\langle w_1, w_2, \ldots,  w_n \mid w_j^2 = 1 = (w_1w_i)^{t_i}, ~1\leq j\leq n,\; 2\leq i\leq n \big\rangle,
\end{equation}
where 
\begin{equation*}
t_i = 
\begin{cases}
m_1 & \text{if}\ 2\leq i \leq k_1+1,\\
m_2 & \text{if}\ k_1+2\leq i \leq k_1+k_2+1,\\
\vdots \\
m_l & \text{if}\ k_1+k_2+\cdots +k_{l-1}+2\leq i \leq k_1+k_2+\cdots +k_l+1=n.
\end{cases}
\end{equation*}

Setting $\mathbb{S}= \{w_1, w_2, \ldots,  w_n \}$, we see that the graph $\mathcal{V}_{(\mathbb{W}, \mathbb{S})}$ is the star shaped graph with $n$ vertices as in Figure 4.
\begin{center}
\begin{tikzpicture}
\draw[fill=black] (3,0) circle (4pt);
\draw[fill=black] (6,0) circle (4pt);
\draw[fill=black] (5,2) circle (4pt);
\draw[fill=black] (5,-2) circle (4pt);
\draw[fill=black] (3,-3) circle (4pt);
\draw[fill=black] (1,-2) circle (4pt);
\draw[fill=black] (3,3) circle (4pt);
\node at (2.6, 0) {1};
\node at (5.7,-0.3) {2};
\node at (4.8,-2.3) {3};
\node at (2.7,-3.3) {4};
\node at (0.7,-2.2) {5};
\node at (2.7,3.3) {n-1};
\node at (4.8,2.3) {n};
\node at (0.5,-0.5) {\textbf{.}};
\node at (0.9, 0.8) {\textbf{.}};
\node at (1.8, 2) {\textbf{.}};
\node at (4.5,0.3) {$t_2$};
\node at (4.4,-1) {$t_3$};
\node at (3.3,-1.5) {$t_4$};
\node at (1.7,-1) {$t_5$};
\node at (2.6,1.5) {$t_{n-1}$};
\node at (3.8,1.2) {$t_n$};
\draw[thick] (6,0) -- (3,0) -- (3,3); 
\draw[thick] (5,2) -- (3,0)-- (5,-2);
\draw[thick] (1,-2) -- (3,0)-- (3,-3);
\end{tikzpicture}
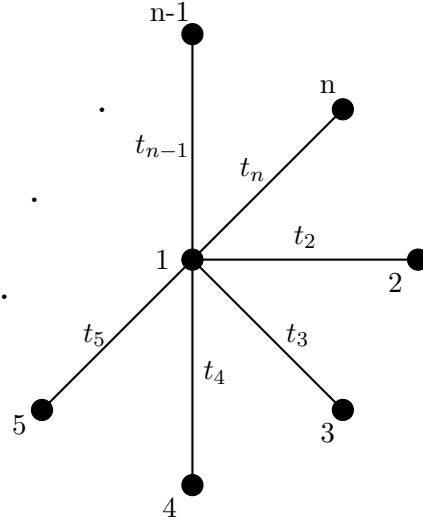
Figure  4.
\end{center}

 Divide the set $\{2, 3, \ldots, n\}$ into $l$ mutually disjoint subsets as follows
\begin{equation*}
A_i = 
\begin{cases}
\{2, 3, \ldots, k_1+1\} & \text{if}\ i = 1,\\
\{k_1+2, k_1+3, \ldots, k_1+k_2+1\}  & \text{if}\ i = 2,\\
\vdots \\
\{k_1+k_2+\cdots +k_{l-1}+2, k_1+k_2+\cdots +k_{l-1}+3, \ldots, k_1+k_2+\cdots +k_l+1\}  & \text{if}\ i = l.
\end{cases}
\end{equation*}

By Corollary \ref{trivial center}, it follows that $\Inn(\mathbb{W}) \cong \mathbb{W}$. The following result characterises $\Aut \big(\Gamma_{(\mathbb{W}, \mathbb{S})}\big)$.

\begin{lemma}
$\Aut\big(\Gamma_{(\mathbb{W}, \mathbb{S})}\big) \cong S_{k_1}\times S_{k_2}\times \cdots \times S_{k_l}$.
\end{lemma}

\begin{proof}
First note that the graph  $\mathcal{V}_{(\mathbb{W}, \mathbb{S})}$  is a subgraph of the complete graph $\Gamma_{(\mathbb{W}, \mathbb{S})}$. Further, any automorphism of $\Gamma_{(\mathbb{W}, \mathbb{S})}$ gives a unique automorphism of $\mathcal{V}_{(\mathbb{W}, \mathbb{S})}$, and any automorphism of $\mathcal{V}_{(\mathbb{W}, \mathbb{S})}$ extends uniquely to an automorphism of $\Gamma_{(\mathbb{W}, \mathbb{S})}$. Thus, $\Aut\big(\Gamma_{(\mathbb{W}, \mathbb{S})}\big)\cong \Aut\big(\mathcal{V}_{(\mathbb{W}, \mathbb{S})}\big)$.
\par
We see that any permutation of $\Pi = \{1, 2, \ldots, n\}$ which extends to an automorphism of $\mathcal{V}_{(\mathbb{W}, \mathbb{S})}$ necessarily fixes the vertex $1$. Due to the labellings on the edges, it follows that a permutation of $\Pi$ which takes an element of $A_i$ to an element of $A_j$ for $i\neq j$ cannot be extended to an automorphism of the graph $\mathcal{V}_{(\mathbb{W}, S)}$. Thus $\Aut(\Gamma_{(\mathbb{W}, \mathbb{S})})$ can be seen as a subgroup of $S_{k_1}\times S_{k_2}\times \cdots \times S_{k_l}$.  Conversely, any permutation of $\Pi$ that keeps each $A_i$ invariant extends to an automorphism of the graph $\mathcal{V}_{(\mathbb{W}, \mathbb{S})}$. This completes the proof of the lemma.
\end{proof}

\begin{remark}
 Let $\alpha\in S_{k_1}\times S_{k_2}\times \cdots \times S_{k_l}$ be a permutation. Then $\alpha$ can be written as $\alpha = \alpha_1 \alpha_2\ldots \alpha_l$, where $\alpha_j \in S_{k_j}$ is a permutation of the set $A_j$ for each $j$. The induced automorphism $\widetilde{\alpha}$ of $\mathbb{W}$ is given on generators  by
\begin{equation}
\widetilde{\alpha}(w_i) = 
\begin{cases}
w_1 & \text{if}\ i=1,\\
w_{\alpha(i)} & \text{if}\ 2\leq i \leq n,
\end{cases}
=
\begin{cases}
w_1 & \text{if}\ i=1,\\
w_{\alpha_j(i)} & \text{if}\ i \in A_j.
\end{cases}
\end{equation}
\end{remark}

Maximal finite parabolic subgroups play an important role in determining the structure of the automorphism group of a Coxeter group. Note that the only maximal finite standard parabolic subgroups of $\mathbb{W}$ (which is clear from the graph  $\mathcal{V}_{(\mathbb{W}, \mathbb{S})}$) are given by
$$\mathbb{W}_i = \langle w_1, w_i\rangle \cong D_{2t_i},$$ 
where $2\leq i \leq n$. In fact, $\mathbb{W}$ is the amalgamated free product of $\mathbb{W}_i$'s over the subgroup $\langle w_1\rangle$. For each integer $2\leq i \leq n$ and $1\leq k< t_i$ with $\gcd (k, t_i)=1$, we define $\theta_i^k$ on the generators of $\mathbb{W}$ as
\begin{equation*}
\theta_i^k(w_j) = 
\begin{cases}
w_j & \text{if}\ j\neq i,\\
w_1(w_1w_i)^k & \text{if}\ j=i.
\end{cases}
\end{equation*}

It is easy to check that $\theta_i^k$ extends to an automorphism of $\mathbb{W}$, and the set
$$C_i = \big\langle \theta_i^k \mid 1\leq k < t_i,\; \gcd(k, t_i) =1 \big\rangle$$ 
is a subgroup of $\Aut(\mathbb{W})$. Note that for each $2\leq i \leq n$, $$C_i \cong \U_{t_i},$$
the multiplicative group of integers mod $t_i$ that are coprime to $t_i$. Setting 
\begin{equation}\label{abelian-aut}
C = C_2  \times C_3 \times \cdots \times C_n,
\end{equation}
we see that 
\begin{equation}\label{C-product-U}
C \cong \prod_{i=1}^{l} (\U_{m_i})^{k_i}
\end{equation}
 is an abelian subgroup of $\Aut(\mathbb{W})$. The following result is an easy observation.

\begin{lemma}\label{invariance-Wi}
Let $\theta \in \Aut(\mathbb{W})$ be an automorphism. Then $\theta \in C$ if and only if $\theta(\mathbb{W}_i) = \mathbb{W}_i$ for each $2\leq i \leq n$.
\end{lemma}

\begin{proof}
If $\theta \in C$, then by construction of $C$, it follows that $\theta(\mathbb{W}_i) = \mathbb{W}_i$ for each $2\leq i \leq n$. Conversely suppose that $\theta(\mathbb{W}_i) = \mathbb{W}_i$ for each $2\leq i \leq n$. Then $$\theta(w_1) \in \bigcap^n_{i=2} \theta(\mathbb{W}_i) = \bigcap^n_{i=2} \mathbb{W}_i = \langle w_1 \rangle,$$
which gives $\theta(w_1) = w_1.$ Fix an integer $2\leq i \leq n$. Note that all the involutions in $\mathbb{W}_i$ are of the form $w_1(w_1w_i)^k$ for some $1\leq k \leq t_i$.  Since automorphisms preserve orders, $\theta(w_i)$ is an involution in $W_i$.Thus, we have $\theta(w_i) = w_1(w_1w_i)^{k_i}$ for some $1\leq k_i \leq t_i$. We claim that $\gcd(k_i, t_i)=1$. If not, say $\gcd(k_i, t_i)= d >1$, then $$\big(\theta(w_1w_i)\big)^{t_i/d}=(w_1w_i)^{k_i(t_i/d)} = \big((w_1w_i)^{t_i}\big)^{k_i/d} =1,$$
which contradicts the fact that $\theta$ is an automorphism. Thus,  $\theta = \prod_{i=2}^n \theta_i^{k_i} \in C$, where each $\gcd(k_i, t_i)=1$, and the proof is complete.
\end{proof}

\begin{lemma}\label{parabolic-subgroups-notconjugate}
$\mathbb{W}_i$ is not conjugate to $\mathbb{W}_j$ for $i\neq j$.
\end{lemma}
\begin{proof}
Suppose that $\mathbb{W}_i = (\mathbb{W}_{j})^w$ for some $w\in \mathbb{W}$. Since all the involutions in $\mathbb{W}_j$ are conjugate to $w_1$, there exists $x\in \mathbb{W}_j$ such that $w_1 = w_1^{xw}$. Then, by Lemma \ref{centralisers of involutions}, we obtain $xw \in \{1, w_1\}$, which gives $w\in \mathbb{W}_j$. Thus, $\mathbb{W}_i = \mathbb{W}_j$, which is true if and only if $i=j$.
\end{proof} 

\begin{lemma}\label{C-intersection-w-hat}
Let $C$ be as in \eqref{abelian-aut}. Then 
$$\Inn(\mathbb{W}) \cap C= \langle \widehat{w_1} \rangle,$$
where $\widehat{w_1}$ is the inner automorphism induced by $w_1$.
\end{lemma}

\begin{proof}
Let $\theta \in \Inn(\mathbb{W}) \cap C$. Then $\theta = \widehat{w}$ for some $w\in \mathbb{W}$. Further, $\theta \in C$ implies that $\theta(w_1) = w_1$, that is, $w w_1=w_1 w$. By Lemma \ref{centralisers of involutions}, we have $w \in \langle w_1 \rangle$, and hence $\Inn(\mathbb{W}) \cap C \leq \langle \widehat{w_1} \rangle$. Conversely, a direct check shows that $\widehat{w_1} = \prod_{i=2}^{n} \theta_i^{t_i-1}\in C\cap \Inn(\mathbb{W})$.
\end{proof}

In view of the preceding lemma, under the isomorphism \eqref{C-product-U}, we identify $\widehat{w_1}\in C$ with $-1 \in \prod_{i=1}^{l} (\U_{m_i})^{k_i}$. We are now ready to prove the main theorem of this section.

\begin{theorem}\label{splitting-by-graph-auto}
Let $\mathbb{W}$ be the group with presentation \ref{star-group}. Then
$$\Aut(\mathbb{W}) = \big ( \Inn(\mathbb{W}) C \big ) \rtimes \Aut \big(\Gamma_{(\mathbb{W}, \mathbb{S})}\big).$$
\end{theorem}

\begin{proof}
Let $\phi \in \Aut(\mathbb{W})$. Since automorphisms preserve orders of elements, by Lemma \ref{involutions are conjugate in TW_n}, there exists $x \in \mathbb{W}$ such that $\widehat{x}\phi(w_1) = w_1$. Note that an automorphism maps maximal finite subgroups to maximal finite subgroups. Further, all maximal finite subgroups of $\mathbb{W}$ are conjugates of $\mathbb{W}_2, \ldots, \mathbb{W}_n$. Thus, for each  $2\leq i \leq n$, there exist $2 \le \alpha(i) \le n$ and elements $x_i \in \mathbb{W}$ such that $\widehat{x}\phi (\mathbb{W}_i) = \mathbb{W}_{\alpha(i)}^{x_i}$. We claim that $x_i \in \mathbb{W}_{\alpha(i)}$. Since $\widehat{x}\phi(w_1) = w_1$, we have $w_1 \in \mathbb{W}_{\alpha(i)}^{x_i}$. Since  $\mathbb{W}_{\alpha(i)} \in \mathcal{W}$, by Lemma \ref{involutions are conjugate in TW_n}, all the involutions in $\mathbb{W}_{\alpha(i)}$ are conjugate to $w_1$. Thus, we have $w_1 = w_1^{wx_i}$ for some $w\in \mathbb{W}_{\alpha(i)}$. By Lemma \ref{centralisers of involutions}, we have $wx_i\in \{1, w_1\}$, which gives $x_i \in \{w^{-1}, w^{-1}w_1\}\subset \mathbb{W}_{\alpha(i)}$. Hence the claim holds and we obtain $\widehat{x}\phi (\mathbb{W}_i) = \mathbb{W}_{\alpha(i)}$. Since $\widehat{x}\phi$ is an automorphism, it maps distinct subgroups to distinct subgroups, and hence $\alpha$ must be a permutation of $\{2, \ldots, n\}$. Moreover, since $\widehat{x}\phi$ preserves orders, we must have $\alpha \in S_{k_1}\times S_{k_2} \times \dots \times S_{k_l}$ seen as subgroup of $S_{\Pi\setminus \{1\}}$. Let $\widetilde{\alpha} \in \Aut \big(\Gamma_{(\mathbb{W}, \mathbb{S})} \big)$ be the graph automorphism induced by $\alpha$. Then we obtain $\widetilde{\alpha}^{-1} \widehat{x}\phi (\mathbb{W}_i)=\mathbb{W}_i$ for each $2 \le i \le n$. Finally, by Lemma \ref{invariance-Wi}, $\widetilde{\alpha}^{-1}\widehat{x}\phi  \in C$. Let us say  $\widetilde{\alpha}^{-1}\widehat{x}\phi = \theta$ for some $\theta \in C$. Now we have 
$$\phi^{-1} = \theta^{-1} \widetilde{\alpha}^{-1}\widehat{x}=  (\theta^{-1} \widetilde{\alpha}^{-1}\widehat{x} \widetilde{\alpha} \theta) (\theta^{-1} \widetilde{\alpha}^{-1}),$$
where $(\theta^{-1} \widetilde{\alpha}^{-1}\widehat{x} \widetilde{\alpha} \theta) \in \Inn(\mathbb{W})$. Hence, we obtain $$\Aut(\mathbb{W}) = \Inn(\mathbb{W}) C \Aut \big(\Gamma_{(\mathbb{W}, \mathbb{S})}\big).$$
\par
Since $\Inn(\mathbb{W})$ is normal in $\Aut(\mathbb{W})$, it follows that $\Inn(\mathbb{W})C$ is a subgroup of $\Aut(\mathbb{W})$. We now claim that $\Inn(\mathbb{W})C$ is normal in $\Aut(\mathbb{W})$. For this, it is sufficient to show that $ \gamma^{-1} C  \gamma \leq C$ for all $ \gamma \in \Aut \big(\Gamma_{(\mathbb{W}, S)}\big)$. In particular, it suffices to check that $\gamma^{-1} C_i \gamma \leq C$, for all $2\leq i \leq n$. Let $\widetilde{\beta} \in \Aut \big(\Gamma_{(\mathbb{W}, \mathbb{S})}\big)$ be induced by the permutation $\beta$ of $\{2, 3, \ldots, n\}$ and $\theta_i^k \in C_i$ for some $1\leq k < t_i$ with $\gcd(k, t_i)=1$. Computing $\widetilde{\beta}^{-1} \theta_i^k \widetilde{\beta}$ on the generators, we get
\begin{equation}\label{graph-action-formula}
\widetilde{\beta}^{-1} \theta_i^k \widetilde{\beta}(w_j) = \widetilde{\beta}^{-1} \theta_i^k \big(w_{\beta(j)}\big) = 
\begin{cases}
w_j & \text{if}\ \beta(j)\neq i,\\
w_1(w_1w_j)^k & \text{if}\ \beta(j)=i.
\end{cases}
\end{equation}
Thus, $\widetilde{\beta}^{-1} \theta_i^k \widetilde{\beta}= \theta_{\beta^{-1}(i)}^k  \in C_{\beta^{-1}(i)}  \leq C$, and the claim holds.
\par 
Finally, we show that $\Inn(\mathbb{W}) C \cap \Aut \big(\Gamma_{(\mathbb{W}, \mathbb{S})}\big) = 1$. Suppose that $1\neq \widetilde{\sigma} = \widehat{w}\theta \in \Inn(\mathbb{W}) C \cap \Aut \big(\Gamma_{(\mathbb{W}, \mathbb{S})}\big)$ for some $\widetilde{\sigma}\in \Aut \big(\Gamma_{(\mathbb{W}, \mathbb{S})}\big)$, $w\in W$ and $\theta\in C$. Since $\sigma \neq 1$, there exist $2\leq i \neq j\leq n$ such that $\sigma(i) = j$. A simple check gives 
$$(W_i)^w=\widehat{w}(W_i) = \widetilde{\sigma} \theta^{-1}(W_i) = \widetilde{\sigma}(W_i) = W_j,$$
which contradicts Lemma \ref{parabolic-subgroups-notconjugate}. Thus, $\Inn(\mathbb{W}) C \cap \Aut \big(\Gamma_{(\mathbb{W}, \mathbb{S})}\big) = 1$, and hence
$$\Aut(\mathbb{W}) = \big ( \Inn(\mathbb{W}) C \big ) \rtimes \Aut \big(\Gamma_{(\mathbb{W}, \mathbb{S})}\big).$$
\end{proof}

\medskip

\section{Splitting of exact sequences of automorphism groups} \label{splitting-sequences}

Theorem \ref{splitting-by-graph-auto} can be viewed as splitting of the short exact sequence 
$$1\rightarrow  \Inn(\mathbb{W}) C\rightarrow \Aut(\mathbb{W}) \rightarrow  \Aut \big(\Gamma_{(\mathbb{W}, \mathbb{S})}\big)\rightarrow 1.$$
Our next objective is to find conditions for the splitting of the natural short exact sequence $$1\rightarrow \Inn(\mathbb{W})\rightarrow \Aut(\mathbb{W}) \rightarrow \Out(\mathbb{W})\rightarrow 1.$$

\begin{lemma}\label{prod-Un-splitting}
Let $p_1, p_2, \ldots, p_n$ be odd primes and $k_1, k_2, \ldots, k_n$ positive integers. Then the short exact sequence of groups
\begin{equation}\label{prod-un-exact-sequence}
1\rightarrow \{\pm 1\} \rightarrow \prod_{i=1}^{n}\U_{p_i^{k_i}} \rightarrow \prod_{i=1}^{n}\U_{p_i^{k_i}}/\{\pm 1\}  \rightarrow 1
\end{equation}
 splits if and only if $p_j\equiv 3 \mod 4$ for some $1\leq j \leq n$.
\end{lemma}

\begin{proof}
We begin by noting that each $\U_{p_i^{k_i}}$ is cyclic group of order $(p_i-1)p_i^{k_i-1}$. Each $p_i$ being an odd prime, either $p_i \equiv 1 \mod 4$ or $p_i \equiv 3 \mod 4$. Let us assume that each $p_i \equiv 1 \mod 4$. Then 4 divides the order of $\U_{p_i^{k_i}}$ for each $i$. It follows from the fundamental theorem of finite abelian groups that $\prod_{i=1}^{n}\U_{p_i^{k_i}}$  cannot have a direct cyclic factor of order $2$. Thus, if the short exact sequence \eqref{prod-un-exact-sequence} splits, then $p_j\equiv 3 \mod 4$ for some $1\leq j \leq n$.
\par

For the converse part, assume without loss of generality that $p_1 \equiv 3 \mod 4$. Then $\frac{(p_1-1)p_1^{k_1-1}}{2}$ is odd, and we have $\U_{p_1^{k_1}}= H_1 \times \{\pm 1\}$, where $H_1$ is cyclic group of order $\frac{(p_1-1)p_1^{k_1-1}}{2}$. Set $H = H_1 \times \prod_{i=2}^{n}\U_{p_i^{k_i}}$. Since $-1\notin H_1$, it follows that $-1 \notin H$. Thus, we obtain $\prod_{i=1}^{n}\U_{p_i^{k_i}} = H \times  \{\pm 1\}$, which is desired.
\end{proof}

\begin{lemma}\label{gen-prod-Un-splitting}
Let $m_1, m_2, \ldots, m_n$ be odd positive integers each greater than one. Then the short exact sequence of groups
\begin{equation*}\label{gen-prod-un-exact-sequence}
1\rightarrow \{\pm 1\} \rightarrow \prod_{i=1}^{n}\U_{m_i} \rightarrow \prod_{i=1}^{n}\U_{m_i}/\{\pm 1\}  \rightarrow 1
\end{equation*}
 splits if and only if there exist some $1\leq j \leq n$ and a prime $p \equiv 3 \mod 4$ such that $p$ divides $m_j$.
\end{lemma}

\begin{proof}
Note that if $m = p_1^{t_1} p_2^{t_2} \cdots p_k^{t_k}$ is the prime factorisation an integer $m$, then $$ \U_{m} \cong  \U_{p_1^{t_1}} \times \U_{p_2^{t_2}} \times \cdots \times \U_{p_k^{t_k}}.$$
The assertion now follows from this fact and Lemma \ref{prod-Un-splitting}.
\end{proof}

\begin{lemma}\label{equivalence-between-C-Un}
The following statements are equivalent.
\begin{enumerate}
\item The sequence $1\rightarrow \langle \widehat{w_1}\rangle \rightarrow C \rightarrow C/ \langle \widehat{w_1}\rangle \rightarrow 1$ splits.
\item The sequence $1\rightarrow \{\pm 1\} \rightarrow   \prod_{i=1}^{l} (\U_{m_i})^{k_i} \rightarrow   \prod_{i=1}^{l} (\U_{m_i})^{k_i}/\{\pm 1\}  \rightarrow 1$ splits.
\end{enumerate}
\end{lemma}

\begin{proof}
The assertion is immediate by observing that $C \cong  \prod_{i=1}^{l} (\U_{m_i})^{k_i}$ and $\langle \widehat{w_1}\rangle \cong \{\pm 1\}$ under this isomorphism.
\end{proof}

If the short exact sequence $$1\rightarrow \langle \widehat{w_1} \rangle \rightarrow C \rightarrow C/\langle \widehat{w_1} \rangle\rightarrow 1$$ splits, then we denote the complement of $\langle \widehat{w_1} \rangle$  in $C$ by $D$.

\begin{theorem}\label{splitting of inn-c-sequence}
Let $\mathbb{W}$ be the group with presentation \ref{star-group}. If there exists a prime $p\equiv 3 \mod 4$ such that $p$ divides $m_i$ for some $1\leq i \leq l$, then the short exact sequence 
\begin{equation}\label{sequence-inn-c-d}
1\rightarrow \Inn(\mathbb{W})\rightarrow \Inn(\mathbb{W})C \rightarrow \Inn(\mathbb{W})C/\Inn(\mathbb{W})\rightarrow 1
\end{equation}
splits. Moreover,
\begin{equation}
\Aut(\mathbb{W})= \big (\Inn(\mathbb{W}) \rtimes D\big ) \rtimes \Aut \big(\Gamma_{(\mathbb{W}, \mathbb{S})}\big).
\end{equation}
\end{theorem}

\begin{proof}
Since $p\equiv 3 \mod 4$ and $p$ divides $m_i$ for some $1\leq i \leq l$, by Lemma \ref{gen-prod-Un-splitting} and Lemma \ref{equivalence-between-C-Un}, it follows that the short exact sequence $1\rightarrow \langle \widehat{w_1} \rangle \rightarrow C \rightarrow C/\langle \widehat{w_1} \rangle\rightarrow 1$ splits. As defined earlier, let $D < C$ be the complement of $\langle \widehat{w_1} \rangle$ in $C$, i.e., $D \cap \langle \widehat{w_1} \rangle = 1$ and $C = D\langle \widehat{w_1} \rangle$. 
By Lemma \ref{C-intersection-w-hat}, since $\Inn(\mathbb{W}) \cap C = \langle \widehat{w_1} \rangle$,  it follows that $\Inn(\mathbb{W}) \cap D = 1$ and $\Inn(\mathbb{W})D = \Inn(\mathbb{W})C$. This shows that the short exact sequence \eqref{sequence-inn-c-d}  splits. Finally, the second assertion follows from Theorem \ref{splitting-by-graph-auto}. 
\end{proof}

\begin{theorem}\label{splitting-out-sequence}
Let $\mathbb{W}$ be the group with presentation \ref{star-group}. If there exists a prime $p\equiv 3 \mod 4$ such that $p$ divides $m_i$ for some $1\leq i \leq l$ with multiplicity of $m_i$ being one, then the short exact sequence
\begin{equation}\label{natural-exact-sequence}
1\rightarrow \Inn(\mathbb{W})\rightarrow \Aut(\mathbb{W}) \rightarrow \Out(\mathbb{W})\rightarrow 1
\end{equation}
 splits. In particular, $\Out (\mathbb{W}) \cong D \rtimes \Aut \big(\Gamma_{(\mathbb{W}, \mathbb{S})}\big)$.
\end{theorem}

\begin{proof}
Without loss of generality, we can assume that $m_1$ is divisible by a prime $p \equiv 3 \mod 4$ and $m_1$ has mutiplicity one. Then, by Lemma \ref{prod-Un-splitting}, it follows that $\{\pm 1\}$ is a direct summand of $\U_{m_1}$ with a complement, say $K$. Since $C_2 \cong \U_{m_1}$, we can write $C_2 = \langle \theta_2^{m_1-1} \rangle\times H$, where $H = \langle \theta_2^k \mid k\in K \rangle$. Also note that $\theta_2^{m_1-1}$ can be identified with $\widehat{w_1}$. Set $D = H \times C_3  \times C_4 \times \cdots \times C_n$. Note that $D$ is a complement of $\langle \widehat{w_1} \rangle$ in $C=C_2 \times C_3  \times \cdots \times C_n$, i.e. $D \cap \langle \widehat{w_1} \rangle = 1$ and $C = D\langle \widehat{w_1} \rangle$. Now, to establish splitting of \eqref{natural-exact-sequence} it suffices to prove that
\begin{itemize}
\item $D\Aut \big(\Gamma_{(\mathbb{W}, \mathbb{S})}\big)$ is a subgroup of $\Aut(\mathbb{W})$,
\item $\Inn(\mathbb{W}) \cap D\Aut \big(\Gamma_{(\mathbb{W}, \mathbb{S})}\big) =1$,
\item $\Aut(\mathbb{W}) = \Inn(\mathbb{W}) D\Aut \big(\Gamma_{(\mathbb{W}, \mathbb{S})}\big)$.
\end{itemize}
\par
Let $\widetilde{\alpha} \in \Aut \big(\Gamma_{(\mathbb{W}, \mathbb{S})}\big)$ be induced by the permutation $\alpha$ of $\{1, 2, \ldots, n\}$. Since the exponent $m_1$ has multiplicity one, it follows that $\alpha$ must fix the vertices $1$ and $2$. Let $\theta_i^k \in C_i$ for some $1\leq k < t_i$ with $\gcd(k, t_i)=1$. Computing $\widetilde{\alpha}^{-1} \theta_i^k \widetilde{\alpha}$ on the generators as in \eqref{graph-action-formula}, we get
\begin{equation*}
\widetilde{\alpha}^{-1} \theta_i^k \widetilde{\alpha}(w_j) = \widetilde{\alpha}^{-1} \theta_i^k \big(w_{\alpha(j)}\big) = 
\begin{cases}
w_j & \text{if}\ \alpha(j)\neq i,\\
w_1(w_1w_j)^k & \text{if}\ \alpha(j)=i.
\end{cases}
\end{equation*}
Since $\alpha$ fixes $1$ and $2$, we have $ \widetilde{\alpha}^{-1} H  \widetilde{\alpha} = H$ and $\widetilde{\alpha}^{-1} C_i  \widetilde{\alpha} \leq \prod_{j=3}^n C_j$ for all $3\leq i\leq n$. This proves that $\widetilde{\alpha}^{-1} D  \widetilde{\alpha} \le D$, i.e. $\Aut \big(\Gamma_{(\mathbb{W}, \mathbb{S})}\big)$ normalises $D$, and hence $D\Aut \big(\Gamma_{(\mathbb{W}, \mathbb{S})}\big)$ is a subgroup of $\Aut(W)$.
\par

Let $\alpha$ be a non-trivial permutation of $\{2, 3, \dots, n\}$, i.e. $\alpha(i) = j$ for some $i \neq j$. Then  $\widetilde{\alpha}(\mathbb{W}_i) = \mathbb{W}_j$. Further, by definition of $C$, it follows that $\widetilde{\alpha}\theta(\mathbb{W}_i) = \mathbb{W}_j$ for each $\theta \in C$. Recall from Lemma \ref{parabolic-subgroups-notconjugate} that $\mathbb{W}_i$ is not conjugate to $\mathbb{W}_j$ for $i\neq j$, and hence $\widetilde{\alpha}\theta\notin \Inn(\mathbb{W})$. Thus, $\Inn(\mathbb{W})  \cap D \Aut \big(\Gamma_{(\mathbb{W}, \mathbb{S})}\big) = \Inn(\mathbb{W}) \cap D = 1.$
\par
The final assertion follows from the fact that $\Aut(W) = \Inn(W) C\Aut \big(\Gamma_{(\mathbb{W}, \mathbb{S})}\big)$ and $C = D\langle \widehat{w_1} \rangle$.
\end{proof}
\medskip


\section{Comparison with braid groups and twin groups}\label{comparision}
This concluding section focus on a special odd Coxeter group $L_n$ of rank $n-1$ and fixed exponent three, which is closely related to the well-known {\it braid group} $B_n$ and {\it twin group} $T_n$. Twin groups can be thought of as planar analogues of braid groups, and have received a great deal of attention in recent works \cite{BarVesSin, GonGutiRoq, Khovanov, MostRoq, NaikNandaSingh1, NaikNandaSingh2}.  
\par

The  symmetric group $S_n$, $n \ge 2$, has a Coxeter presentation with generators $\{\tau_1, \ldots, \tau_{n-1} \} $ and defining relations
\begin{enumerate}
\item $\tau_i^2 = 1$ for  $1\leq i \leq n-1$;
\item Braid relations: $\tau_i\tau_{i+1}\tau_i=\tau_{i+1}\tau_i\tau_{i+1}$ for $1\leq i \leq n-2$;
\item Far commutativity relations: $\tau_i\tau_j=\tau_j\tau_i$ for $\mid i - j\mid \geq 2$.
\end{enumerate}
\par

By omitting all relations of type (1), (2) or (3) at a time from the preceding presentation of $S_n$, we get presentations of the braid group $B_n$, the twin group $T_n$ and the group $L_n$ as follows:
\begin{align}\label{definition-Ln}
\nonumber B_n &= \big{\langle} \sigma_1, \sigma_2, \ldots, \sigma_{n-1} ~\mid~ \sigma_j\sigma_k=\sigma_k\sigma_j,~~ \sigma_i\sigma_{i+1}\sigma_i=\sigma_{i+1}\sigma_i\sigma_{i+1}~~\text{for all} ~~\mid j - k\mid \geq 2~~\text{and} ~~1\leq i \leq n-2 \big{\rangle}.&\\
\nonumber T_n &= \big{\langle} s_1, s_2, \ldots, s_{n-1} ~\mid~ s_k^2 = 1,~~s_is_j=s_js_i~~ \text{for all}~~1\leq k \leq n-1~~\text{and} ~~  \mid i - j\mid \geq 2 \big{\rangle}.&\\
\nonumber L_n &= \big{\langle} y_1, y_2, \ldots, y_{n-1} ~\mid~ y_j^2 = 1,~~y_iy_{i+1}y_i=y_{i+1}y_i y_{i+1} ~\text{for all} ~~1\leq j \leq n-1~~\text{and} ~~1\leq i \leq n-2 \big{\rangle}.&\\
\end{align}

The groups $B_n$, $T_n$ and $L_n$ fit into the commutative diagram of surjections as in Figure 5, where $\mathcal{U}_n$ is the universal Coxeter group of rank $(n-1)$, $F_{n-1}$ is the free group of rank $(n-1)$, $\mathcal{R}_n$ is a right angled Artin group with only far commutativity relations (2), and $\mathcal{A}_n$ is an Artin group with only braid relations (3). We compare automorphism groups, $R_\infty$-property and (co)-Hopfian property of $B_n$, $T_n$ and $L_n$.
\begin{small}
$$
\xymatrix{
& & F_{n-1} \ar@{->>}[d]   \ar@{->>}@/^4pc/[ddddr] \ar@{->>}@/_4pc/[ddddl]& & \\
& & \mathcal{U}_n  \ar@{->>}[rd]  \ar@{->>}[ld]& & \\
& T_n  \ar@{->>}[rd]& & L_n   \ar@{->>}[ld] & \\
&   & S_n &   & \\
& \mathcal{R}_n  \ar@{->>}[uu]  \ar@{->>}[rd] &  & \mathcal{A}_n \ar@{->>}[uu]  \ar@{->>}[ld] & \\
& & B_n  \ar@{->>}[uu] & &  \textrm{Figure 5.}}
$$ 
\end{small}

\subsection{Pure subgroups}
There are natural surjective homomorphisms $B_n \to S_n$ and $T_n \to S_n$. The kernel of the former homomorphism is denoted by $P_n$ and is called the {\it pure braid group}, and the kernel of the latter homomorphism is denoted by $PT_n$ and is referred as the {\it pure twin group}. Similarly, for $L_n$, let $\pi:L_n \to S_n$ be the natural surjection, i.e $\pi(y_i)= \tau_i$ for all $i$, and let $PL_n$ denote its kernel.
\par

It is well-known that $P_n$ is not a subgroup of $B_n'$ for $n \ge 2$ and $PT_n$ is not a subgroup of $T_n'$ for $n \ge 3$. Further, $P_n$ is characteristic in $B_n$ for $n \ge 2$ \cite{Dyer-Gross} and $PT_n$ is not characteristic in $T_n$ for $n \ge 4$ \cite[Proposition 6.11]{NaikNandaSingh1}. For $PL_n$, we have 

\begin{proposition}\label{PLn-subgroup-commutator}
The following holds:
\begin{enumerate}
\item $PL_n$ is a subgroup of $L_n'$ for  $n \ge 2$.
\item $PL_n$ is not a characteristic subgroup of $L_n$ for  $n \ge 4$.
\end{enumerate}
\end{proposition}

\begin{proof}
 Note that $PL_2 = 1 =PL_3$. By Thoerem \ref{generalised theorem for commutator}, $L_n'$ is an index two subgroup of $L_n$. Therefore, if $PL_n \nleq L_n'$, then $L_n' PL_n = L_n$. Then $\pi(L_n)= \pi(L_n' PL_n)=\pi(L_n')=A_n$, which is a contradiction to the surjectivity of $\pi$. Hence, $PL_n$ is a subgroup of $L_n'$, which proves assertion (1).
\par

Set $x_i = y_i y_{i+1}$ for  $1\leq i \leq n-2$ and $z=y_2$. Then, using  Theorem \ref{generalised theorem for commutator},  we can write
$$L_n = L_n' \rtimes \langle z \rangle,$$
where $L_n'=\langle x_1, x_2, \ldots, x_{n-2} ~|~x_i^3=1~\textrm{for}~1\leq i \leq n-2  \rangle$. Note that the action of $\langle z \rangle$ on $L_n'$ is given by
$$ x_1^z = x_1^{-1},~ x_2^z = x_2^{-1}$$
and
$$x_j^z = x_2 x_3 \cdots x_{j-1}x_j^{-1}x_{j-1}^{-1}\cdots x_2^{-1}~\textrm{for}~3\leq j \leq n-2.$$

Consider the map $\phi$ defined by $x_1 \mapsto x_1^{-1}$, $x_i\mapsto x_i$ for $2\leq i \leq n-2$ and $z\mapsto z$. With the preceding presentation of $L_n$, it can be checked easily that $\phi$ extends to an automorphism of $L_n$. Now, the element $(x_1x_2)^2 \in PL_n$, but $\phi (x_1x_2)^2 = (x_1^{-1}x_2)^2 \notin  PL_n$, since $\pi (x_1^{-1}x_2)^2 = \big ( (2 3)(1 2)(2 3)(3 4) \big )^2 = (1 3 4)$. Thus, $PL_n$ is not a characteristic subgroup of $L_n$ for $n \ge 4$, which is assertion (2).
\end{proof}

It is known that $PT_n$ is  free if and only if $3 \le n \le 5$ \cite{BarVesSin, GonGutiRoq}. On the other hand, the pure braid group $P_n$ is free if and only if $n=2$. We need the following result of Kurosch \cite{Hall, Kurosch} for the description of $PL_n$.

\begin{lemma}\label{subgroup-of-freeproducts}
Let $G = *_\nu A_\nu$ be a free product of groups and $H$ a subgroup of $G$. Then $H$ itself is a free product, more precisely, $$H = F * \big(*_j ( x_j^{-1}U_jx_j) \big),$$
where $F$ is a free group and each $x_j^{-1}U_jx_j$ is some conjugate of a subgroup $U_j$ of one of the free factors $A_\nu$ of $G$.
\end{lemma}

\begin{proposition}\label{PKn-rank}
$PL_n$ is a non-abelian free group of finite rank if and only if $n \ge 4$.
\end{proposition}

\begin{proof}
A direct check shows that the elements $(y_1y_3)^2$ and $y_2(y_3y_1)^2y_2$ of $PL_n$ do not commute, and hence $PL_n$ is non-abelian for $n \ge4$. It follows from Theorem \ref{generalised theorem for commutator} that $L_n'$ is a free product of $(n-2)$ copies of the cyclic group of order three. Further, by Proposition \ref{PLn-subgroup-commutator}, $PL_n$ is a subgroup of  $L_n'$. Now, by Lemma \ref{subgroup-of-freeproducts}, $PL_n$ is a non-abelian free group since it contains no conjugates of generators of the cyclic groups of order three. Finally, it follows from Reidemeister-Schreier Theorem \cite[Theorem 2.6]{Magnus1966} that $PL_n$ is a non-abelian free group of finite rank.
\end{proof}

\begin{remark}
Determining the precise rank of $PL_n$ seems a computationally challenging problem. While the structure of the pure braid group $P_n$ is well-known \cite{KasselTuraev}, a precise description of the pure twin group $PT_n$ is also unknown for 
$n \ge 7$ \cite[Remark 3.5]{NaikNandaSingh2}.
\end{remark}

Since $PL_n$ is normal in $L_n$, there is a natural homomorphism 
$$\phi_n : L_n\cong \Inn(L_n)\rightarrow \Aut(PL_n),$$
obtained by restricting the inner automorphisms. By Proposition \ref{PKn-rank}, $PL_n$ is a non-abelian free group of finite rank for $n \ge 4$. This gives a representation of $L_n$ into the automorphism group of a free group.

\begin{theorem}
For $n\geq 4$, the representation
$$\phi_n : L_n \rightarrow \Aut(PL_n)$$
is faithful.
\end{theorem}

\begin{proof}
Note that $\ker(\phi_n) = \C_{L_n}(PL_n)$. We first claim that $\C_{L_n'}(PL_n)=1$. By Theorem \ref{generalised theorem for commutator}, we know that $L_n'$ is a free product of $n-2$ copies of the cyclic group of order three. Thus, it follows that for any $1 \neq x\in L_n'$, $\C_{L_n'}(x)$ is either the cyclic group of order three or the infinite cyclic group. If $1 \neq y\in \C_{L_n'}(PL_n)$, then $PL_n \leq \C_{L_n'}(y)$, which contradicts Proposition \ref{PKn-rank}. Now suppose that $1 \neq z\in \C_{L_n}(PL_n)$. Then $z^2 \in \C_{L_n'}(PL_n)=1$, and hence $z^2 = 1$. By Lemma \ref{centralisers of involutions}, we get $\C_{L_n}(z)= \langle z \rangle$. But $PL_n \le \C_{L_n}(z)$, which is a contradiction. Hence, $\C_{L_n}(PL_n)=1$, and the map $\phi_n$ is injective.
\end{proof}

\medskip

\subsection{Automorphism groups}
The structure of the automorphism group of braid groups is well-known \cite[Theorem 19]{Dyer-Gross}. More precisely, for $n \ge 2$,
 $$\Aut(B_n) \cong \Inn(B_n) \rtimes \Out(B_n),$$
where  $\Out(B_n)\cong  \mathbb{Z}/2\mathbb{Z}$ is generated by the automorphism given by $\sigma_i  \mapsto \sigma_i^{-1}$ for all $1 \le i \le n-1$. Similarly, a complete description of the automorphism group of twin groups is known \cite[Theorem 6.1]{NaikNandaSingh1}. For $n\geq 3$, $$\Aut(T_n)\cong \Inn(T_n)\rtimes \Out(T_n),$$
 where
\begin{equation*}
\Out(T_n) \cong 
\begin{cases}
\mathbb{Z}_2 & \text{if}\ n=3,\\
S_3 & \text{if}\ n=4,\\
D_8 & \text{if}\ n \geq 5.\\
\end{cases}
\end{equation*}
By solution of the isomorphism problem for $\mathcal{TW}$, the group $L_n$, $n\geq 2$, is isomorphic to the special case of the group $\mathbb{W}$ defined in \eqref{star-group}). More precisely, $L_n \cong \mathbb{W}_n$, where we set
\begin{equation}\label{Ln-iso-wn}
\mathbb{W}_n =\big\langle w_1, w_2, \ldots,  w_{n-1} \mid w_j^2 = 1 = (w_1w_i)^3, ~1\leq j\leq n-1,\; 2\leq i\leq n-1 \big\rangle.
\end{equation}
Note that $L_2 \cong \mathbb{Z}/2\mathbb{Z}$ and $L_3 \cong S_3$. Setting $\mathbb{S}=\{w_1, w_2, \ldots,  w_{n-1}\}$, by Theorem \ref{splitting of inn-c-sequence},  we have 
$$\Aut(L_n) \cong \big(\Inn(\mathbb{W}_n)\rtimes D \big) \rtimes\Aut \big(\Gamma_{(\mathbb{W}_n, \mathbb{S})}\big),$$
where 
$$D\cong (\U_3)^{n-2}/ \{\pm 1\}\cong (\mathbb{Z}/2\mathbb{Z})^{n-3}$$
and
$$\Aut \big(\Gamma_{(\mathbb{W}_n, \mathbb{S})}\big) \cong S_{n-2}.$$ 

Thus, we obtain

\begin{theorem}
Let $L_n$ be as defined in \eqref{definition-Ln}. Then
\begin{equation*}
\Aut(L_n) \cong 
\begin{cases}
L_3 & \text{if}\ n=3,\\
\big(L_n \rtimes (\mathbb{Z}/2\mathbb{Z})^{n-3} \big) \rtimes S_{n-2} & \text{if}\ n\geq 4,
\end{cases}
\end{equation*}
where 
\begin{equation*}
\Out(L_n) \cong 
\begin{cases}
1 & \text{if}\ n= 3\\
(\mathbb{Z}/2\mathbb{Z})^{n-3} \rtimes S_{n-2} & \text{if}\ n\geq 4.
\end{cases}
\end{equation*}
\end{theorem}

An automorphism of a group is said to be \textit{normal} if it maps every normal subgroup onto itself. It is known that every normal automorphism of $B_n$ and that of  $T_n$ is inner for $n \geq 2$ (see \cite{Neshchadim2016} and \cite[Proposition 6.14]{NaikNandaSingh2}).

\begin{theorem}
Every normal automorphism of $L_n$ is inner for $n \ge 2$.
\end{theorem}

\begin{proof}
The assertion is obvious for $n=2, 3$. As in \eqref{Ln-iso-wn}, $L_n \cong \mathbb{W}_n$. By Theorem \ref{splitting-by-graph-auto}, we have 
$$\Aut(L_n) \cong \big ( \Inn(\mathbb{W}_n) C \big ) \rtimes \Aut \big(\Gamma_{(\mathbb{W}_n, \mathbb{S})}\big).$$
Since all inner automorphisms are normal, it suffices to check normality for automorphisms of the form $\phi = \theta \widetilde{\alpha}$, where $\theta \in C$ and $\widetilde{\alpha} \in \Aut \big(\Gamma_{(\mathbb{W}_n, \mathbb{S})}\big)$. If $\widetilde{\alpha} \neq 1$, then the inducing permutation $\alpha$ of $\{2, 3, \ldots, n-1\}$ is non-trivial, say, $\alpha (i) = j$ for some $i \neq j$. Taking $N$ to be the normal closure of the element $w_1w_i$ in $\mathbb{W}_n$, we see that $w_1w_i \in N$, but $\phi(w_1w_i) =\theta \widetilde{\alpha}(w_1w_i) = w_1\theta \widetilde{\alpha}(w_i)=w_1\theta(w_j)=(w_1 w_j)^k \not\in N$, where $k=1, 2$.  Thus, if $\widetilde{\alpha}\neq1$, then $\phi$ cannot be a normal automorphism.
\par
Now suppose that  $\phi = \prod_{i=2}^{n-1} \theta_i^{k_i} \in C$, where each $k_i \in \{1, 2\}$. If each $k_i=1$, then $\phi$ is the identity automorphism. Similarly, if each $k_i=2$, then $\phi$ is precisely the inner automorphism induced by $w_1$. Suppose that at least two $k_i$'s are distinct, say,  $k_2=1$ and $k_3=2$. Taking $H$ to be the normal closure of $w_2w_3$ in $\mathbb{W}_n$, we see that $w_2w_3 \in H$, but $\phi(w_2w_3) \not\in H$. Thus, every normal automorphism of $\mathbb{W}_n$ must be inner.
\end{proof}
\medskip

\subsection{$R_{\infty}$ and (co)-Hopfianity}
Next, we discuss  $R_{\infty}$-property of groups in the family $\mathcal{TW}$, in particular, groups $L_n$. A group $G$ is said to have $R_{\infty}$-property if it has infinitely many $\phi$-twisted conjugacy classes for each automorphism $\phi$ of $G$,  where two elements $x, y \in G$ lie in the same $\phi$-twisted conjugacy class if there exists $g \in G$ such that $x = gy\phi(g)^{-1}$. The study of $R_{\infty}$-property of groups has attracted a lot of attention in recent years, see, for example,  \cite{Daciberg-Sankaran2019, GoncalvesSankaranWong, Nasybullov-comm2019, Nasybullov-group-2019}. It is known that braid groups $B_n$ have $R_{\infty}$-property for all $n \ge 3$ \cite{Fel'shtyn2010}. Further, it was proved recently in \cite{NaikNandaSingh2} that twin groups $T_n$ also have $R_{\infty}$-property for all $n \ge 3$.

\begin{theorem}\label{thm-R-infinity}
Let $W\in \mathcal{TW}$ be a Coxeter group admitting an odd connected Coxeter system of rank $n \ge 3$. Then $W$ satisfy the  $R_{\infty}$-property.  In particular, $L_n$ has the $R_{\infty}$-property for $n \ge 3$.
\end{theorem}

\begin{proof}
By Lemma \ref{generalised theorem for commutator}, we have the short exact sequence $1 \to W' \to W \to \mathbb{Z}/2\mathbb{Z} \to 1$. By Theorem \ref{generalised theorem for commutator}, $W'$ is a free product of finite cyclic groups. By \cite[Lemma 2]{GoncalvesSankaranWong}, a free product of non-trivial finite groups, in particular $W'$, has $R_{\infty}$ property. Finally, since $W'$ is characteristic in $W$, by \cite[Lemma 2.2(ii)]{MubeenaSankaran}, it follows that $R_{\infty}$ property.
\end{proof}
\medskip

Recall that a group is co-Hopfian (respectively Hopfian) if every injective (respectively surjective) endomorphism is an automorphism. These properties are known to be closely related to $R_{\infty}$-property. See, for example, \cite[Lemma 2.3]{MubeenaSankaran}. Braid groups $B_n$ are known to be Hopfian being residually finite \cite[Chapter I, Corollary 1.22]{KasselTuraev} and are not co-Hopfian for $n \ge 2$ \cite{BellMargalit}. It is well-known that Coxeter groups, in particular $T_n$ and $L_n$, are Hopfian \cite[Theorem C, p.55]{Brown1989}. The twin groups $T_n$ are co-Hopfian only for $n = 2$ \cite[Theorem 4.1]{NaikNandaSingh2}. We conclude with the following result on co-Hopfianity of the group $\mathbb{W}$.

\begin{theorem}\label{thm-cohopfian}
Let $W\in \mathcal{TW}$ be a Coxeter group admitting an odd connected Coxeter system of rank $n \ge 2$. Then
$W$ is co-Hopfian. In particular, $L_n$ is co-Hopfian for $n \ge 2$.
\end{theorem}

\begin{proof}
Without loss of generality we can take $W= \mathbb{W}$ as defined in \eqref{star-group}. Let $\phi$ be an injective homomorphism of $\mathbb{W}$. Then $\phi(w_1)$ is an involution. By Lemma \ref{involutions are conjugate in TW_n}, there exists a $w \in \mathbb{W}$ such that $\widehat{w}\phi(w_1) = w_1$. Setting $\theta= \widehat{w}\phi$, it now suffices to show that $\theta$ is surjective.
\par

Let $\mathbb{W}_i= \langle w_1, w_i \rangle$ for $2\leq i \leq n$. For each fixed $i$, since $\theta$ is injective, we have $\mathbb{W}_i \cong \theta(\mathbb{W}_i)$. Thus,  $\theta(\mathbb{W}_i)= (\mathbb{W}_j)^x$ for some $x\in W$ and $2\leq j\leq n$. As in the proof of Theorem \ref{splitting-by-graph-auto}, we can show that $x\in \mathbb{W}_j$, and hence $\theta(\mathbb{W}_i) = \mathbb{W}_j$. Since $\theta$ is injective, $\theta(\mathbb{W}_{i_1}) \neq \theta(\mathbb{W}_{i_2})$ for $i_1\neq i_2$, and hence
$$\big\{\mathbb{W}_i \mid 2\leq i \leq n \big\} = \big\{\theta(\mathbb{W}_i) \mid 2\leq i \leq n \big\}.$$
But, $\theta(\mathbb{W}) \geq \big\langle \theta(\mathbb{W}_i) \mid 2\leq i \leq n \big\rangle = \big\langle \mathbb{W}_i \mid 2\leq i \leq n \big\rangle = \mathbb{W}$, which proves that $\theta$ is surjective.
\end{proof}

Elements of braid groups $B_n$ are well-known to be represented by geometric braids in the 3-space and are deeply related to knot theory \cite{KasselTuraev}. Similarly, elements of twin groups $T_n$ can be represented by strands of line segments on the plane without triple intersections \cite{Khovanov} and relates to study of immersed circles on the two sphere. We do not know whether there is a similar topological interpretation of elements of $L_n$.

\begin{problem}
Does there exist a topological interpretation of the group $L_n$ analogous to that of $B_n$ and $T_n$?
\end{problem}
\medskip

\begin{ack}
The authors are grateful to Prof. Bernhard M\"{u}hlherr for useful comments on an earlier version of this paper, in particular, for pointing out his work \cite{MuhlherrWeidmann2000} which contains a solution of the isomorphism problem for the family $\mathcal{TW}$. Tushar Kanta Naik is supported by the Institute Post Doctoral Fellowship of IISER Mohali. Mahender Singh is supported by the Swarna Jayanti Fellowship grants DST/SJF/MSA-02/2018-19 and SB/SJF/2019-20/04, and the MATRICS Grant MTR/2017/000018. 
\end{ack}

\end{document}